\documentclass[12pt,reqno]{amsart}
\usepackage{amssymb,amsmath,amsthm,enumerate,verbatim}

\DeclareMathOperator{\sign}{sign}
\DeclareMathOperator{\Ran}{Ran}

\DeclareMathOperator{\Tr}{Tr}

\DeclareMathOperator{\diag}{diag}

\DeclareMathOperator{\supp}{supp}

\DeclareMathOperator*{\slim}{s-lim}
\DeclareMathOperator*{\wlim}{w-lim}
\DeclareMathOperator{\Det}{Det}

\renewcommand\Im{\hbox{{\rm Im}}\,}

\newcommand{\abs}[1]{\lvert#1\rvert}
\newcommand{\Abs}[1]{\left\lvert#1\right\rvert}
\newcommand{\norm}[1]{\lVert#1\rVert}
\newcommand{\Norm}[1]{\left\lVert#1\right\rVert}

\newcommand{\bbT}{{\mathbb T}}
\newcommand{\bbR}{{\mathbb R}}
\newcommand{\bbC}{{\mathbb C}}
\newcommand{\bbN}{{\mathbb N}}

\newcommand{\bbD}{{\mathbb D}}

\newcommand{\hm}{{H^\sharp}}

\newcommand{\calH}{{\mathcal H}}
\newcommand{\calK}{{\mathcal K}}

\newcommand{\calN}{\mathcal{N}}
\newcommand{\calC}{\mathcal{C}}

\newcommand{\calB}{\mathcal{B}}

\newcommand{\calQ}{\mathcal{Q}}
\newcommand{\Sch}{\mathbf{S}}
\numberwithin{equation}{section}

\theoremstyle{plain}
\newtheorem{theorem}{\bf Theorem}[section]
\newtheorem*{theorem*}{Theorem 1.1$'$}
\newtheorem{lemma}[theorem]{\bf Lemma}
\newtheorem{proposition}[theorem]{\bf Proposition}

\newtheorem{corollary}[theorem]{\bf Corollary}

\theoremstyle{definition}

\theoremstyle{remark}
\newtheorem*{remark*}{\bf Remark}
\newtheorem{remark}[theorem]{\bf Remark}
\newtheorem{example}[theorem]{\bf Example}

\newcommand{\wt}{\widetilde}
\newcommand{\eps}{\varepsilon}

\newcommand{\ac}{\text{(ac)}}

\begin{document}

\title[Scattering theory and singular integrals]{Scattering theory and Banach space valued singular integrals}

\author{Alexander Pushnitski}
\address{Department of Mathematics, King's College London, Strand, London, WC2R~2LS, U.K.}
\email{alexander.pushnitski@kcl.ac.uk}

\author{Alexander Volberg}
\address{Department of Mathematics, Michigan State University, 
East Lansing, MI 48824, U.S.A.}
\email{volberg@math.msu.edu}



\begin{abstract}
We give a new sufficient condition for existence and completeness
of wave operators in abstract scattering theory. 
This condition generalises both trace class and smooth approaches 
to scattering theory. 
Our construction is based on estimates for the Cauchy transforms
of operator valued measures. 
\end{abstract}

\maketitle

%%%%%%%%%%%%%%%%%%%%%%%%%%%%%%%%%%%%%%%%%%%%%%%%%
%%%%%%%%%%%%%%%%%%%%%%%%%%%%%%%%%%%%%%%%%%%%%%%%%
\section{Introduction}\label{sec.a}
%%%%%%%%%%%%%%%%%%%%%%%%%%%%%%%%%%%%%%%%%%%%%%%%%
%%%%%%%%%%%%%%%%%%%%%%%%%%%%%%%%%%%%%%%%%%%%%%%%%

%%%%%%%%%%%%%%%%%%%%%%%%%%%%%%%%
\subsection{Trace class and smooth versions of scattering theory}\label{sec.a1}
%%%%%%%%%%%%%%%%%%%%%%%%%%%%%%%%

Let $H_0$ and $H_1$ be self-adjoint operators in a separable Hilbert space $\calH$. 
In mathematical scattering theory one is interested in establishing the
unitary equivalence of the  absolutely continuous (a.c.) parts of $H_0$ and $H_1$. 
This unitary equivalence is effected by the \emph{wave operators}
\begin{equation}
W_\pm(H_1,H_0)
=
\slim_{t\to\pm\infty} e^{itH_1}e^{-itH_0}P^{\ac}_{H_0},
\label{a1}
\end{equation}
where $P^{\ac}_{H_0}$ denotes the orthogonal projection onto the a.c. subspace
of the operator $H_0$, and $\slim$ denotes the limit in the strong operator
topology. 
Recall (see e.g. \cite{Yafaev})
that if the wave operators \eqref{a1} exist, they are readily seen to be
isometric on the subspace $\Ran P^{\ac}_{H_0}$, and they intertwine the operators
$H_0$ and $H_1$: 
$$
H_1 W_\pm(H_1,H_0) = W_\pm(H_1,H_0) H_0.
$$
If the wave operators are \emph{complete}, i.e. if the condition 
$$
\Ran W_+(H_1,H_0)
=
\Ran W_-(H_1,H_0)
=
\Ran P^{\ac}_{H_1}
$$
holds true, then each one of the two operators $W_\pm(H_1,H_0)$  provides
a unitary equivalence between the a.c. parts of $H_0$ and $H_1$. 
If both $W_\pm(H_0,H_1)$ and $W_\pm(H_1,H_0)$ exist, then
all of these wave operators are complete.

There are two fundamental results on existence and completeness of 
wave operators: the Kato-Rosenblum theorem and the Kato smoothness
theorem; let us recall these results.
For $1\leq p<\infty$, we denote by $\Sch_p$ the usual Schatten-von Neumann class
with the norm $\norm{\cdot}_p$.

%%%%%%%%%%%%%%%%%%%
\begin{theorem}[Kato-Rosenblum \cite{Kato1,Rosenblum}]
\label{thm-KR}
%%%%%%%%%%%%%%%%%%%
If the difference $H_1-H_0$ belongs to the trace class $\Sch_1$, then the wave
operators $W_\pm(H_1,H_0)$ exist and are complete. 
\end{theorem}

In order to state the Kato smoothness theorem, 
we write the perturbation
$H_1-H_0$ in a factorised form, i.e. we assume that 
\begin{equation}
H_1=H_0+V, \qquad V=G^*JG, 
\label{a2}
\end{equation}
where $G$ and $J=J^*$ are bounded operators in $\calH$. 
Of course, the simplest case of such a factorisation is
$G=\abs{V}^{1/2}$ and $J=\sign(V)$.
%%%%%%%%
\begin{remark*}
In fact, $G$ does not have to be bounded, but for simplicity 
we would like to avoid the discussion of this  technical issue.
We also note that it is often convenient to allow $G$ 
to act from $\calH$ to another  Hilbert space $\calK$, but 
in order to accommodate this case, it suffices simply to make 
straightforward notational changes. 
\end{remark*}
%%%%%%%%
We recall that an operator
$G$ is called $H_0$-smooth, if for some constant $C>0$  and for all intervals
$\delta\subset\bbR$ the estimate
\begin{equation}
\norm{GE_{H_0}(\delta)G^*}\leq C\abs{\delta}
\label{a3}
\end{equation}
holds true; here $E_{H_0}(\cdot)$ 
is the projection-valued spectral measure of $H_0$ and $\abs{\delta}$
is the Lebesgue measure of $\delta$. 

%%%%%%%%%%%%%%%%%%%
\begin{theorem}[Kato \cite{Kato2}]
\label{thm-K}
%%%%%%%%%%%%%%%%%%%
If $H_1$ is represented in the form \eqref{a2}, where
$G$ is both $H_0$-smooth and $H_1$-smooth, then the wave operators
$W_\pm(H_1,H_0)$ exist and are complete. 
\end{theorem}

\begin{remark}\label{rmk.a1}
In applications, one usually knows the spectral representation of $H_0$ 
in a more or less explicit form but the spectral representation of $H_1$
is unknown.
Thus it is difficult to apply Theorem~\ref{thm-K} directly. 
In practice, one usually assumes a stronger condition than \eqref{a3}, 
such as the existence, H\"older continuity and compactness of 
the derivative
\begin{equation}
\frac{d}{dx}GE_{H_0}(-\infty,x) G^*.
\label{a3a}
\end{equation}
Then, by a version of perturbation theory, using the Fredholm analytic alternative, 
one can establish the  \emph{local} $H_1$-smoothness of $G$ on any compact sub-interval
of the set $\bbR\setminus Z$, where $Z$ is a closed set of 
measure zero. This suffices for the proof of existence and completeness
of the wave operators. 
For the details, see e.g. \cite[Section~4.6]{Yafaev}.
\end{remark}

If $V$ is factorised as $V=G^*JG$, the Kato-Rosenblum theorem
can be equivalently stated as follows:

\begin{theorem*}\label{thm-KR2}
If $H_1$ is represented in the form \eqref{a2},  where $G\in \Sch_2$, then 
the  wave operators
$W_\pm(H_1,H_0)$ exist and are complete. 
\end{theorem*}
Indeed, $G\in\Sch_2$ implies that $V=G^*JG\in\Sch_1$; 
on the other hand, if $V\in\Sch_1$, then $\abs{V}^{1/2}\in\Sch_2$, and 
taking $G=\abs{V}^{1/2}$ and $J=\sign(V)$, we obtain the representation
\eqref{a2} with $G\in\Sch_2$. Thus, Theorem~\ref{thm-KR} 
is equivalent to Theorem~1.1$'$.

The hypotheses of 
Theorems~1.1$'$ and \ref{thm-K} are very different, and neither of them 
implies the other one. 
In order to compare these hypotheses,
let us consider

\begin{example}\label{exa.e1}
Let $\calH=L^2((0,1),\calN)$, where $\calN$ is an auxiliary Hilbert space, 
and let $H_0$ be the multiplication by an independent variable in this space:
$(H_0 f)(x)=xf(x)$ for all $f\in\calH$. 
Let $\calK$ be another Hilbert space, and 
let $J=J^*$ be a bounded operator in $\calK$; we set
$H_1=H_0+G^*JG$, where the operator $G:\calH\to\calK$ satisfies the following assumption. 
We suppose that $G$ is of the form
$$
Gf=\int_0^1 G(x)f(x)dx
$$
with $G(x)\in\calB(\calN,\calK)$ for every $x\in(0,1)$;
here and in what follows $\calB(\calN,\calK)$ is the Banach space of all
bounded operators from $\calN$ to $\calK$. 
The hypothesis $G\in\Sch_2$ of Theorem~1.1$'$ can be rewritten as
\begin{equation}
\norm{G}^2_{\Sch_2(\calH,\calK)}
=
\int_0^1\norm{G(x)}_{\Sch_2(\calN,\calK)}^2 dx
<\infty.
\label{e1}
\end{equation}
The hypothesis of $H_0$-smoothness of $G$ is equivalent to 
\begin{equation}
\sup_{x\in(0,1)}\norm{G(x)}_{\calB(\calN,\calK)}
<\infty. 
\label{e2}
\end{equation}
The hypothesis of $H_1$-smoothness of $G$ is difficult to state in an explicit form. 
One usually assumes in addition that $G(x)$ is compact for all $x$ and that the function 
$(0,1)\ni x\mapsto G(x)\in\calB(\calN,\calK)$ 
is H\"older continuous with some positive H\"older exponent. 
Then the derivative \eqref{a3a} 
exists and equals $G(x)G(x)^*$
and thus one can check the local $H_1$-smoothness of $G$ as outlined above. 

If at least one of the Hilbert spaces $\calN$, $\calK$ is finite dimensional, 
then $\calB(\calN,\calK)=\Sch_2(\calN,\calK)$ and so \eqref{e1}
follows from \eqref{e2}. 
In general, none of the two assumptions \eqref{e1}, \eqref{e2}
implies the other one, since none of the two spaces
\begin{equation}
L^2((0,1), \Sch_2), 
\quad
L^\infty((0,1),\calB)
\label{e3}
\end{equation}
is contained in the other one. 
\end{example}

No simple generalisation which would imply both Theorem~1.1$'$ and Theorem~\ref{thm-K} is known. 
It is known, for example, that the trace class $\Sch_1$ in Theorem~\ref{thm-KR}
cannot be replaced by any class $\Sch_p$ with $p>1$ (see \cite[Theorem X.2.3]{Kato3}):

%%%%%%%%%%%%%%%%%%%
\begin{theorem}[Weil-von Neumann-Kuroda]
\label{thm-WvNK}
%%%%%%%%%%%%%%%%%%%
For any self-adjoint operator $H_0$, any $p>1$ and any $\eps>0$ there
exists an operator $V\in\Sch_p$ with $\norm{V}_p<\eps$ such that
$H_1=H_0+V$ has a pure point spectrum.
\end{theorem}
Thus, if $H_0$ has some a.c. spectrum, one can always choose
$V\in\Sch_p$, $p>1$, in such a way as to ``destroy'' this spectrum, 
and then of course the a.c. parts of $H_0$ and $H_1$ are not 
unitarily equivalent to one another.

%%%%%%%%%%%%%%%%%%%%%%%%%%%%%%%%
\subsection{Main results}\label{sec.a2}
%%%%%%%%%%%%%%%%%%%%%%%%%%%%%%%%
We assume that $H_1$ is represented in the form \eqref{a2}
with bounded operators $G$ and $J=J^*$. 
Our main result is
%%%%%%%%%%%%%%%%%%%
\begin{theorem}
\label{thm.a3}
%%%%%%%%%%%%%%%%%%%
Suppose that for some $p<\infty$ and some non-negative
$\sigma$-finite measure $\nu_0$ on $\bbR$, one has
\begin{equation}
\norm{GE_{H_0}(\delta)G^*}_p\leq \nu_0(\delta)
\label{a4}
\end{equation}
for all intervals $\delta\subset\bbR$. 
Then the wave operators
$W_\pm(H_1,H_0)$ exist and are complete. 
\end{theorem}
Obviously, if \eqref{a4} holds true for some $p=p_0$, then 
it holds true for all $p>p_0$. 

In order to compare Theorem~\ref{thm.a3} with 
the Kato-Rosenblum theorem, we note that under the hypothesis
$G\in \Sch_2$ one has
$$
\norm{GE_{H_0}(\delta)G^*}_1
=
\Tr(GE_{H_0}(\delta)G^*)
=
\nu_0(\delta),
$$
where $\nu_0(\delta)$ is a non-negative scalar finite measure on $\bbR$. 
Thus, Theorem~\ref{thm.a3} can be regarded as a generalization 
of the Kato-Rosenblum theorem (extension from $p=1$ to any $p<\infty$). 

On the other hand, Theorem~\ref{thm.a3} is ``almost" a generalization 
of the Kato smoothness theorem. Indeed, if one replaces the $\Sch_p$ norm
in \eqref{a4} by the operator norm, one obtains condition \eqref{a3}
with $\nu_0$ being the Lebesgue measure.

%%%%%%%%%%%%%%%%added on 23 June: begin%%%%%%%%%%%%%%
In order to illustrate this, let us return to Example~\ref{exa.e1}. 
%%%%%%%%%%%%%%%
\begin{corollary}\label{cr.a4}
%%%%%%%%%%%%%%%
In Example~\ref{exa.e1}, suppose that for some $p<\infty$ one has
\begin{equation}
\int_0^1\norm{G(x)}_{2p}^2 dx<\infty.
\label{e4}
\end{equation}
Then the hypothesis of Theorem~\ref{thm.a3} is satisfied and therefore the wave
operators $W_\pm(H,H_0)$ exist and are complete. 
\end{corollary}
Thus, in Example~\ref{exa.e1}, it suffices to check the inclusion
$G(\cdot)\in L^2((0,1);\Sch_{2p})$ with some $p<\infty$; 
it is instructive to compare this with spaces \eqref{e3}. 
Of course, a sufficient condition for \eqref{e4} is 
$G(\cdot)\in L^{2p}((0,1);\Sch_{2p})$ with $p>1$. 
\begin{proof}
First let us check that
\begin{equation}
\norm{G}_{2p}^2
\leq
\int_0^1\norm{G(x)}_{2p}^2 dx; 
\label{e5}
\end{equation}
note that in the l.h.s. of \eqref{e5} we have the norm in $\Sch_p(\calH,\calK)$, 
and in the r.h.s. it is the norm in $\Sch_p(\calN,\calK)$. 
It suffices to check this inequality on the dense set of 
continuous operator valued functions $G:[0,1]\to \Sch_2(\calN,\calK)$. 
Diagonalising $G(x)$ for each $x\in(0,1)$ and choosing 
orthonormal bases in $\calN$ and $\calK$, one  reduces the problem to 
the case when $\calN=\calK=\ell^2$ and $G(x)$ is a diagonal 
operator in $\ell^2$, $G(x)=\diag\{g_1(x),g_2(x),\dots\}$. 
Then we have
\begin{gather}
\norm{G}_{2p}^2
=
\left(\sum_n\left(\int_0^1\abs{g_n(x)}^2 dx\right)^{p}\right)^{1/p},
\label{e6}
\\
\int_0^1\norm{G(x)}_{2p}^2 dx
=
\int_0^1\left(\sum_n\abs{g_n(x)}^{2p}\right)^{1/p}dx.
\label{e7}
\end{gather}
Now \eqref{e5} follows from the convexity of the norm in $\ell^p$. 
More precisely,
denoting $h_n(x)=\abs{g_n(x)}^2$ and 
interpreting the sequence $\{h_n(x)\}=h(x)$ as an element of $\ell^p$, 
we obtain
\begin{multline*}
\left(\sum_n\left(\int_0^1\abs{g_n(x)}^2 dx\right)^{p}\right)^{1/p}
=
\left(\sum_n\left(\int_0^1 h_n(x) dx\right)^{p}\right)^{1/p}
=
\Norm{\int_0^1 h(x)dx}_{\ell^{p}}
\\
\leq
\int_0^1 \norm{h(x)}_{\ell^{p}}dx
=
\int_0^1 \left(\sum_n h_n(x)^{p}\right)^{1/p} dx
=
\int_0^1 \left(\sum_n \abs{g_n(x)}^{2p}\right)^{1/p} dx.
\end{multline*}
Together with \eqref{e6} and \eqref{e7}, this proves \eqref{e5}.

In the above argument we can replace $(0,1)$ by any subinterval; 
this yields
$$
\norm{GE_{H_0}(\delta)G^*}_p
=
\norm{GE_{H_0}(\delta)}_{2p}^2
\leq
\int_\delta\norm{G(x)}_{2p}^2 dx
=
\nu_0(\delta), 
$$
where $d\nu_0(x)=\norm{G(x)}_{2p}^2 dx$ is, by assumption, a finite measure. 
Thus, we obtain \eqref{a4}. 
\end{proof}

In Kato-Rosenblum and Kato smoothness theorems, the hypotheses are
symmetric with respect to interchanging $H_0$ and $H_1$. 
In Theorem~\ref{thm.a3}, the hypothesis involves only the spectral measure
of $H_0$; this is convenient in applications, where one usually knows 
the spectral measure of $H_0$ but not the spectral measure of $H_1$. 
However, for completeness we also give a result symmetric with respect 
to interchanging $H_0$ and $H_1$:

%%%%%%%%%%%%%%%%%%%
\begin{theorem}
\label{thm.a4}
%%%%%%%%%%%%%%%%%%%
Suppose that for some non-negative
$\sigma$-finite measures $\nu_0$ and $\nu_1$ on $\bbR$, one has
\begin{equation}
\norm{GE_{H_0}(\delta)G^*}\leq \nu_0(\delta), 
\quad
\norm{GE_{H_1}(\delta)G^*}\leq \nu_1(\delta), 
\label{a5}
\end{equation}
for all intervals $\delta\subset\bbR$. 
Then the wave operators
$W_\pm(H_1,H_0)$ exist and are complete. 
\end{theorem}
Note that the norms in \eqref{a5} are the usual operator norms, so in contrast 
with Theorem~\ref{thm.a3}, here we do not require even compactness of 
the operators in \eqref{a5}. 
Clearly, Theorem~\ref{thm.a4} is a direct generalization of the Kato smoothness theorem.

\begin{remark}
One can ask whether the hypothesis of Theorem~\ref{thm.a3} 
ensures that a similar condition holds for the spectral  measure of $H_1$, 
i.e. whether \eqref{a4} ensures the existence of a measure $\nu_1$ such 
that 
$$
\norm{GE_{H_1}(\delta)G^*}_p\leq \nu_1(\delta)
$$
holds true for all intervals $\delta$. 
The answer to this is negative for $p>1$. 
Indeed, let $H_0=0$ and let $V\in\Sch_p$, $p>1$, be a self-adjoint
operator such that all eigenvalues of $V$ are non-degenerate and 
$V\notin\Sch_1$; let $G=\abs{V}^{1/2}$,
$J=\sign(V)$. Then \eqref{a4} holds true with $\nu_0=\norm{V}_p \delta_0$,
where $\delta_0$ is the delta-measure supported at zero. 
On the other hand, it is easy to see that the total variation of the $\Sch_p$-valued
measure $GE_{H_1}(\cdot)G^*$ on $\bbR$  is infinite, 
and so there is no measure $\nu_1$ satisfying the above condition. 
\end{remark}

%%%%%%%%%%%%%%%%%%%%%%%%%%%%%%%%
\subsection{Method of proof}\label{sec.a4}
%%%%%%%%%%%%%%%%%%%%%%%%%%%%%%%%
We use the following notation:
\begin{equation}
R_j(z)=(H_j-zI)^{-1}, 
\quad 
B_j(z)=GR_j(z)G^*,
\quad
j=0,1,
\qquad \Im z\not=0.
\label{a11}
\end{equation}
It is well understood that the proof of existence 
and completeness of wave operators essentially reduces to the proof of existence
of limits $B_0(\lambda+i0)$, $B_1(\lambda+i0)$ in an appropriate sense
for a.e. $\lambda\in\bbR$; see Proposition~\ref{prp.d1} for a precise statement. 
Assume here for simplicity of discussion that the measure
$\nu_0$ in the hypothesis of Theorem~\ref{thm.a3}
is finite rather than $\sigma$-finite. 
In order to establish the existence of the limits $B_0(\lambda+i0)$,
we represent $B_0(z)$ as the Cauchy transform of the 
$\Sch_p$-valued measure
\begin{equation}
\mu_0(\delta)=GE_{H_0}(\delta)G^*, 
\quad \delta\subset \bbR;
\label{a12}
\end{equation}
that is, we write
$$
B_0(z)=\int_{\bbR}\frac{d\mu_0(t)}{t-z}, 
\quad \Im z>0.
$$
Then the $\Sch_p$-valued measure $\mu_0$ has a 
\emph{finite total variation}. We recall that if $X$ is a Banach space 
with the norm $\norm{\cdot}$, 
then an $X$-valued measure $\mu$ on $\bbR$ 
is said to have a finite total variation, if the supremum
\begin{equation}
\norm{\mu}(\bbR)
:=
\sup \textstyle \sum_n \norm{\mu(\delta_n)}
\label{a13}
\end{equation}
over all finite collections of disjoint intervals $\delta_n\subset \bbR$ 
is finite. 
Our key technical result concerns the Cauchy transforms of the 
Banach space valued measures of finite variation. A remarkable class of 
Banach spaces is given by the spaces with the \emph{UMD property}, 
see \cite{Bourgain1,Bourgain2,Burk}. These are the spaces $X$ such that the Cauchy transform
(and many other singular integral transforms) are bounded in $L^2(\bbR,X)$, 
see \eqref{b13}.
Any Hilbert space possesses this property, see e.g. \cite{Stein}.
For any finite $p>1$, the Schatten class $\Sch_p$ possesses the 
UMD property, see \cite{deF}.
%%%%%%%%%%%%%%%%%%%
\begin{theorem}\label{thm.a7}
%%%%%%%%%%%%%%%%%%%
Let $X$ be a Banach space which possesses the UMD property, and 
let $\mu$ be an $X$-valued measure on $\bbR$ which has a finite
total variation. 
Let $\calC\mu$ be the Cauchy transform of $\mu$, 
\begin{equation}
\calC \mu(z)
=
\int_{\bbR}\frac{d\mu(t)}{t-z}, 
\quad \Im z>0,
\label{a14}
\end{equation}
and let $T^<\mu$ be the non-tangential maximal function for $\calC\mu$:
\begin{equation}
(T^<\mu)(\lambda)
=
\sup\{\norm{\calC\mu(x+iy)}: y>0, \, \abs{x-\lambda}<y\}.
\label{a15}
\end{equation}
Then $T^<\mu$ belongs to the weak class $L^{1,\infty}(\bbR)$ and 
the estimate
\begin{equation}
\sup_{s>0} \  s \abs{\{\lambda \in\bbR: (T^<\mu)(\lambda)>s\} }\leq C \norm{\mu}(\bbR)
\label{a16}
\end{equation}
holds true with a constant $C$ which depends only on the space $X$. 
\end{theorem}
We prove Theorem~\ref{thm.a7} in Section~\ref{sec.b}, which forms
the central part of the paper. 
In order to prove Theorem~\ref{thm.a3}, we will apply Theorem~\ref{thm.a7}
to the $\Sch_p$-valued measure $\mu_0$ given by \eqref{a12}. 
From the inclusion $T^<\mu\in L^{1,\infty}(\bbR)$ we derive the existence
of boundary values $B_0(\lambda+i0)$ in $\Sch_p$ for a.e. $\lambda\in\bbR$.
Finally, a simple argument (borrowed from the book \cite{Yafaev})
involving the modified determinant $\Det_q(I+B_0(z)J)$ ($q\geq p$ is any integer) 
and Privalov's uniqueness theorem allows us to prove 
the invertibility of $I+B_0(\lambda+i0)J$ for a.e. $\lambda$; 
by the identity 
\begin{equation}
B_1(z)=(I+B_0(z)J)^{-1}B_0(z)
\label{a17}
\end{equation}
this yields  existence of boundary values 
of $B_1(z)$.  This is done in Section~\ref{sec.c}.

In order to prove Theorem~\ref{thm.a4}, we apply Theorem~\ref{thm.a7}
to the $\calH$-valued measures 
\begin{equation}
GE_{H_0}(\delta)G^*\psi,
\quad
GE_{H_1}(\delta)G^*\psi,
\qquad 
\delta\subset\bbR,
\label{a18}
\end{equation}
for an arbitrary fixed element $\psi\in\calH$.
Since the Hilbert space $\calH$ possesses the UMD property,  
this yields the  existence of the limits 
\begin{equation}
B_0(\lambda+i0)\psi,
\quad
B_1(\lambda+i0)\psi,
\qquad
\text{ a.e. $\lambda\in\bbR$,}
\label{a19}
\end{equation}
and then the conclusion of Theorem~\ref{thm.a4} can be derived
from the standard technique of stationary scattering 
theory. 

The assumption that $\mu_0$ has a bounded variation is crucial for 
our construction. Indeed, in \cite{Naboko}
for any $p>1$ Naboko has constructed an $\Sch_p$-valued measure
$\mu$ such that the limits $\lim_{\eps\to+0}(\calC\mu)(\lambda+i\eps)$
do not exist even in the weak sense on a set of $\lambda\in\bbR$ of the 
full Lebesgue measure. 
Naboko's measure $\mu$ is a countably additive function from Borel
sets on $\bbR$ to $\Sch_p$, but it fails to have a finite total variation.
More precise interesting results related to this issues can also be 
found in \cite{Naboko}.

%%%%%%%%%%%%%%%%%%%%%%%%%%%%%%%%
\subsection{Local versions}\label{sec.a3}
%%%%%%%%%%%%%%%%%%%%%%%%%%%%%%%%

It is sometimes convenient to compare the spectral structure
of $H_0$ and $H_1$ locally on some interval $\Delta\subset\bbR$. 
Below we give local versions of Theorems~\ref{thm.a3}, \ref{thm.a4};
they do not really require any considerable modification of the technique.
First we need to recall the definitions related to local wave operators.

If $\Delta\subset\bbR$ is an open interval, the local wave operators
$W_\pm(H_1,H_0;\Delta)$ are defined by 
\begin{equation}
W_\pm(H_1,H_0;\Delta)
=
\slim_{t\to\pm\infty} e^{itH_1}e^{-itH_0}E_{H_0}(\Delta)P^{\ac}_{H_0}.
\label{a6}
\end{equation}
The local wave operators \eqref{a6} are called complete, if 
$$
\Ran W_+(H_1,H_0;\Delta)
=
\Ran W_-(H_1,H_0;\Delta)
=
\Ran(E_{H_1}(\Delta) P^{\ac}_{H_1}).
$$
If the wave operators \eqref{a6} exist and are complete, then the a.c. 
parts of the restrictions $H_0|_{\Ran E_{H_0}(\Delta)}$ and 
$H_1|_{\Ran E_{H_1}(\Delta)}$ are unitarily equivalent (see \cite{Yafaev}).
As in the case of the global wave operators, if both 
$W_\pm(H_0,H_1;\Delta)$ and $W_\pm(H_1,H_0;\Delta)$ exist, then
all of these wave operators are complete.

%%%%%%%%%%%%%%%%%%%
\begin{theorem}
\label{thm.a5}
%%%%%%%%%%%%%%%%%%%
Let $H_1=H_0+G^*JG$, where 
the operators $G$ and $J=J^*$ are bounded and $GR_0(z)$
is compact for $\Im z\not=0$. 
Let $p<\infty$ and let $\Delta\subset \bbR$ be an open interval. 
Suppose that for some non-negative finite measure $\nu_0$ on $\Delta$ 
one has
\begin{equation}
\norm{GE_{H_0}(\delta)G^*}_{p}\leq \nu_0(\delta)
\label{a9}
\end{equation}
for all intervals $\delta\subset\Delta$. 
Then the local wave operators
$W_\pm(H_1,H_0;\Delta)$ exist and are complete. 
\end{theorem}

%%%%%%%%%%%%%%%%%%%
\begin{theorem}
\label{thm.a6}
%%%%%%%%%%%%%%%%%%%
Let $H_1=H_0+G^*JG$, where 
the operators $G$ and $J=J^*$ are bounded and $GR_0(z)$
is compact for $\Im z\not=0$. 
Suppose that for some 
non-negative finite measures $\nu_0$ and $\nu_1$ on an interval $\Delta\subset\bbR$, 
one has
\begin{equation}
\norm{GE_{H_0}(\delta)G^*}\leq \nu_0(\delta), 
\quad
\norm{GE_{H_1}(\delta)G^*}\leq \nu_1(\delta), 
\label{a10}
\end{equation}
for all intervals $\delta\subset\Delta$. 
Then the local wave operators
$W_\pm(H_1,H_0;\Delta)$ exist and are complete. 
\end{theorem}

%%%%%%%%%%%%%%%%%%%%%%%%%%%%%%%%%%%%%%%%%%%%%%%%%
%%%%%%%%%%%%%%%%%%%%%%%%%%%%%%%%%%%%%%%%%%%%%%%%%
\section{Singular integrals for Banach space valued measures}\label{sec.b}
%%%%%%%%%%%%%%%%%%%%%%%%%%%%%%%%%%%%%%%%%%%%%%%%%
%%%%%%%%%%%%%%%%%%%%%%%%%%%%%%%%%%%%%%%%%%%%%%%%%

%%%%%%%%%%%%%%%%%%%%%%%%%%%%%%%%
\subsection{Definitions and background}\label{sec.b1}
%%%%%%%%%%%%%%%%%%%%%%%%%%%%%%%%
For a fixed $x\in\bbR$ and $r>0$, we denote
$B(x,r)=(x-r,x+r)$. 
For a non-negative scalar measure $\nu$ on $\bbR$, 
the Hardy-Littlewood maximal function is defined by 
\begin{equation}
M\nu(x)
=
\sup_{r>0}
\frac1{2r}\nu(B(x,r)).
\label{b2}
\end{equation}
Next, for a scalar function $g:\bbR\to[0,\infty)$ and an exponent 
$\beta\in(0,1)$, we set
\begin{equation}
M_\beta g(x)
=
\bigl(M\abs{g}^\beta(x)\bigr)^{1/\beta}
=
\sup_{r>0}
\left(\frac1{2r}\int_{\abs{y-x}<r}\abs{g(y)}^\beta dy\right)^{1/\beta}.
\label{b3}
\end{equation}
We recall that the quasi-norm in the weak space $L^{1,\infty}(\bbR)$ 
is defined by 
$$
\norm{f}_{L^{1,\infty}}
=
\sup_{t>0} t\abs{\{x\in\bbR: \abs{f(x)}>t\}}.
$$

\begin{proposition}\label{prp.b1}
\begin{enumerate}[\rm (i)]
\item
For any finite scalar non-negative measure $\nu$, 
\begin{equation}
\norm{M\nu}_{L^{1,\infty}}
\leq 
3\nu(\bbR).
\label{b4}
\end{equation}
\item
For any $\beta\in(0,1)$, the (non-linear) operator $M_\beta$
is bounded in $L^{1,\infty}(\bbR)$: 
\begin{equation}
\norm{M_\beta g}_{L^{1,\infty}}
\leq
\frac{6^{1/\beta}}{1-\beta}
\norm{g}_{L^{1,\infty}}.
\label{b5}
\end{equation}
\end{enumerate}
\end{proposition}
See e.g. \cite{SteinS} for \eqref{b4} and \cite{NVT} for \eqref{b5}.

Next, we discuss Banach space measures.
Let $X$ be a Banach space. An $X$-valued measure on $\bbR$ is defined, as usual, 
as a countably additive (the series must converge in the norm of $X$) 
map from the collection of Borel sets on $\bbR$ to $X$. 
If $\mu$ is an $X$-valued measure on $\bbR$ and $\Delta\subset\bbR$ 
is an interval, then the total variation of $\mu$ on $\Delta$  (similarly to \eqref{a13}) is
\begin{equation}
\norm{\mu}(\Delta)
:=
\sup \textstyle \sum_n \norm{\mu(\delta_n)}
\label{b1}
\end{equation}
where the supremum is taken 
over all finite collections of disjoint intervals $\delta_n\subset \Delta$.
If the total variation of $\mu$ on $\bbR$ is finite, then 
$\norm{\mu}(\cdot)$ is itself a non-negative scalar finite measure on $\bbR$. 
Conversely, if for an $X$-valued measure $\mu$ and a finite scalar measure
$\nu$ on $\bbR$ one has
$$
\norm{\mu(\delta)}\leq \nu(\delta)
$$
for any interval $\delta\subset\bbR$, then $\mu$ has a finite total variation on $\bbR$.

Let $\mu$ be an $X$-valued measure on $\bbR$ of a finite total variation. 
The Hilbert transform of $\mu$ is defined by 
\begin{equation}
H\mu(x)
=
\lim_{r\to0+}
\int_{\abs{x-y}>r}\frac{d\mu(y)}{y-x},
\label{b6} 
\end{equation}
if this limit exists (in the norm of $X$). 
If $f:\bbR\to X$ is a function in $L^p(\bbR,X)$, $p<\infty$, then 
its Hilbert transform is defined as in \eqref{b6} with $\mu(x)=f(x)dx$:
$$
H(fdx)(x)
=
\lim_{r\to0+}
\int_{\abs{x-y}>r}\frac{f(y)dy}{y-x}.
$$
A Banach space $X$ is said to possess the UMD property,
if the Hilbert transform is bounded in $L^2(\bbR,X)$, i.e. if the estimate
\begin{equation}
\int_\bbR \norm{H(fdx)(x)}^2 dx
\leq
C_X
\int_\bbR \norm{f(x)}^2 dx
\label{b13}
\end{equation}
holds true with some constant $C_X$. It is known \cite{deF}
that $\Sch_p$ possesses the UMD property for any $1<p<\infty$. 

Finally, we need some notation related to dyadic intervals on $\bbR$. 
Such intervals, i.e. the intervals of the form $(j2^{-n},(j+1)2^{-n}]$, 
will be denoted by $Q$, $Q_j$, etc. For a dyadic interval $Q=(a,b]$, its
``parent" is denoted by $\widehat Q$, its center is denoted by $c(Q)$,
its length is denoted by $\abs{Q}$,  
and $2Q=(c(Q)-\abs{Q},c(Q)+\abs{Q}]$ 
is the ``scaled up" version  of $Q$. The complement of a set $A$ is
denoted by $A^c$.

%%%%%%%%%%%%%%%%%%%%%%%%%%%%%%%%
\subsection{Hilbert transforms of simple measures}\label{sec.b2}
%%%%%%%%%%%%%%%%%%%%%%%%%%%%%%%%
Throughout the rest of this section, $X$ is a Banach space with a UMD property; 
we will denote the norm in $X$ by $\norm{\cdot}$, and let $C_X$ be
the constant from \eqref{b13}.
We start by considering simple $X$-valued measures $\mu$ on $\bbR$,
i.e. the measures of the form
\begin{equation}
\mu=\sum_{i=1}^N \delta_{x_i}e_i,
\label{b7}
\end{equation}
where $N$ is finite, $\delta_{x_i}$ is a $\delta$-measure supported 
at $x_i\in\bbR$ and $e_i$ are elements of $X$. 
Of course, such measures have a finite total variation.
For simple measures, the Hilbert transform \eqref{b6} is obviously
well defined for all $x\not=x_i$. 

%%%%%%%%%%%%%%%%
\begin{lemma}\label{lma.b1}
%%%%%%%%%%%%%%%%
Let $\mu$ be a simple measure
of the form \eqref{b7}.
Then the function $\norm{H\mu(x)}$ belongs to $L^{1,\infty}(\bbR)$ and 
the estimate
\begin{equation}
\norm{H\mu}_{L^{1,\infty}}\leq (30+4C_X)\norm{\mu}(\bbR)
\label{b8}
\end{equation}
holds true. 
\end{lemma}
\begin{proof}
The key idea is to approximate the measure $\mu$ by an 
absolutely continuous measure $f(x)dx$ with an appropriately
constructed $X$-valued function $f$, and then to use the 
UMD property \eqref{b13}.
When approximating $\mu$ by $f(x)dx$, we must make sure that the 
$L^{1,\infty}$ norm of the error term is controlled by the total variation of $\mu$. 

1.
Fix $s>0$; consider the set of all dyadic intervals $Q$ such that 
\begin{equation}
\frac{\norm{\mu}(Q)}{\abs{Q}}>s.
\label{b9}
\end{equation}
Clearly, if $Q$ is a sufficiently large interval which contains $\supp \mu$, 
then \eqref{b9} fails. This shows that the set of dyadic intervals
which satisfy \eqref{b9} has finitely many maximal elements 
(if ordered by inclusion). 
Denote these maximal elements by $\{Q_\ell\}_{\ell=1}^L$.
For each dyadic interval $Q$, let us define the function $f_Q$ by 
\begin{equation}
f_Q(x)
=
\begin{cases}
\frac1{\abs{Q}}\mu(Q), 
&x\in Q,
\\
0,
&x\notin Q.
\end{cases}
\label{b10}
\end{equation}
Thus, for any $x\in\bbR$, 
$$
\norm{f_Q(x)}
\leq
\frac1{\abs{Q}}\norm{\mu(Q)}
\leq
\frac1{\abs{Q}}\norm{\mu}(Q).
$$
Next, for each maximal interval $Q_\ell$, let $\widehat Q_\ell$ be
its parent interval. By the maximality of $Q_\ell$, we have
$$
\frac{\norm{\mu}(\widehat Q_\ell)}{\abs{\widehat Q_\ell}}\leq s.
$$
Using this, we get that for any $x\in\bbR$, 
\begin{equation}
\norm{f_{Q_\ell}(x)}
\leq
\frac1{\abs{Q_\ell}}\norm{\mu}(Q_\ell)
\leq
\frac1{\abs{Q_\ell}}\norm{\mu}(\widehat Q_\ell)
=
\frac{2}{\abs{\widehat Q_\ell}}\norm{\mu}(\widehat Q_\ell)
\leq
2s.
\label{b11}
\end{equation}

2. 
Define the function 
$$
f(x)=\sum_{\ell=1}^L f_{Q_\ell}(x)
$$
and set $\nu=\mu-fdx$.
We have
$$
H\mu
=
H(fdx)+H\nu,
$$
and therefore for each $x\in\bbR$ 
$$
\norm{H\mu(x)}
\leq
\norm{H(fdx)(x)}+\norm{H\nu(x)}.
$$
It follows that for any $s>0$, 
\begin{equation}
\Omega_s^\mu\subset\Omega^f_{s/2}\cup \Omega^\nu_{s/2},
\label{b11a}
\end{equation}
where
$$
\Omega_s^\mu =\{x: \norm{H\mu(x)}>s\},
\quad
\Omega_s^f =\{x: \norm{H(fdx)(x)}>s\}.
$$
Our aim is to prove that both $\abs{\Omega^f_s}$ and 
$\abs{\Omega^\nu_s}$ can be estimated above via
$s^{-1}\norm{\mu}(\bbR)$. 

3. 
Consider the set $\Omega_s^f$. 
Using the estimates \eqref{b13} and \eqref{b11}, we get
\begin{multline*}
s^2 \abs{\Omega_s^f}
\leq
\int_{\Omega_s^f} \norm{(H(fdx))(x)}^2 dx
\leq
C_X
\int_\bbR \norm{f(x)}^2 dx
\leq
2s C_X 
\int_\bbR \norm{f(x)} dx
\\
\leq
2s C_X 
\sum_{\ell=1}^L
\int_\bbR \norm{f_{Q_\ell}(x)} dx
\leq
2s C_X 
\sum_{\ell=1}^L \Norm{\mu(Q_\ell)}
\leq
2s C_X 
\norm{\mu}(\bbR),
\end{multline*}
which yields the estimate
\begin{equation}
\abs{\Omega_s^f}
\leq 
\frac{2 C_X}{s}\norm{\mu}(\bbR).
\label{b12}
\end{equation}

4. 
Consider the set $\Omega_s^\nu$.
Let us split this set as follows:
\begin{equation}
\Omega_s^\nu
=
\bigl(\Omega_s^\nu\cap(\cup_{\ell=1}^L 2Q_\ell)\bigr)
\cup
\bigl(\Omega_s^\nu \cap (\cup_{\ell=1}^L 2Q_\ell)^c\bigr).
\label{b14}
\end{equation}
Let us separately estimate the Lebesgue measure of each of the two sets
in the union in the r.h.s. of \eqref{b14}. 
The first set is easy to deal with; using the definition of $Q_\ell$
(see \eqref{b9}), we get:
\begin{equation}
\abs{\Omega_s^\nu\cap(\cup_{\ell=1}^L 2Q_\ell)}
\leq
\abs{\cup_{\ell=1}^L 2Q_\ell}
\leq
2\sum_{\ell=1}^L \abs{Q_\ell}
\leq
2\sum_{\ell=1}^L \frac{\norm{\mu}(Q_\ell)}{s}
\\
\leq
\frac{2}{s}\norm{\mu}(\bbR).
\label{b14a}
\end{equation}
Let us consider the second set in the r.h.s. of \eqref{b14}. 
Our next aim is to check the estimate
\begin{equation}
\int_{(\cup_{\ell=1}^L 2Q_\ell)^c}\norm{H\nu(x)}dx
\leq
4\pi \norm{\mu}(\bbR).
\label{b15}
\end{equation}

5.
For all $\ell$, let us denote by $\nu_\ell$ the restriction of $\nu$ onto $Q_\ell$.
By the definition of $\nu$, we have $\nu_\ell(Q_\ell)=0$ and 
$\norm{\nu_\ell}(Q_\ell)\leq 2\norm{\mu}(Q_\ell)$ for all $\ell$. 
Let us fix $\ell$ and estimate $\norm{H\nu_\ell(x)}$ pointwise 
for $x\in(2Q_\ell)^c$.
Denoting $c_\ell=c(Q_\ell)$ and using $\nu_\ell(Q_\ell)=0$, we get
\begin{multline*}
H\nu_\ell(x)
=
\int_{Q_\ell}\frac{1}{x-y}d\nu_\ell(y)
=
\int_{Q_\ell}\left(\frac1{x-y}-\frac1{x-c_\ell}\right)d\nu_\ell(y)
\\
=
\int_{Q_\ell}\frac{y-c_\ell}{(x-y)(x-c_\ell)}d\nu_\ell(y).
\end{multline*}
We need to estimate the integrand in the last expression.
For all $y\in Q_\ell$, $x\in (2Q_\ell)^c$ we have the elementary estimates
$$
\abs{y-c_\ell}\leq \frac12\abs{Q_\ell},
\quad
\abs{x-y}\geq\frac12\abs{x-c_\ell},
\quad
\abs{x-c_\ell}\geq \abs{Q_\ell}, 
$$
and therefore
\begin{equation}
\frac{\abs{y-c_\ell}}{\abs{x-y}\abs{x-c_\ell}}
\leq
\frac{\frac12\abs{Q_\ell}}{\frac12\abs{x-c_\ell}^2}
\leq
\frac{2\abs{Q_\ell}}{\abs{x-c_\ell}^2+\abs{Q_\ell}^2}.
\label{b15a}
\end{equation}
It follows that 
$$
\norm{H\nu_\ell(x)}
\leq
2\int_\bbR 
\frac{\abs{Q_\ell}}{\abs{x-c_\ell}^2+\abs{Q_\ell}^2}
d\norm{\nu_\ell}(y)
\leq
4\norm{\mu}(Q_\ell) 
\frac{\abs{Q_\ell}}{\abs{x-c_\ell}^2+\abs{Q_\ell}^2}
$$
for all $x\in (2Q_\ell)^c$. 
Now for any $x\in(\cup_{\ell=1}^L 2Q_\ell)^c$ we can sum up the previous
estimate over $\ell$:
\begin{multline*}
\int_{(\cup_{\ell=1}^L 2Q_\ell)^c}\norm{H\nu(x)}dx
\leq
\sum_{\ell=1}^L
\int_{(\cup_{\ell=1}^L 2Q_\ell)^c} \norm{H\nu_\ell(x)}dx
\leq
\sum_{\ell=1}^L
\int_{(2Q_\ell)^c} \norm{H\nu_\ell(x)}dx
\\
\leq
4\sum_{\ell=1}^L\norm{\mu}(Q_\ell) 
\int_{(2Q_\ell)^c}
\frac{\abs{Q_\ell}}{\abs{x-c_\ell}^2+\abs{Q_\ell}^2} dx
\leq
4\pi \norm{\mu}(\bbR),
\end{multline*}
as claimed in \eqref{b15}.

6. 
By the Chebyshev inequality, \eqref{b15} yields the estimate
\begin{equation}
\abs{\Omega_s^\nu \cap (\cup_{\ell=1}^L 2Q_\ell)^c}
=
\abs{\{x\in (\cup_{\ell=1}^L 2Q_\ell)^c: \norm{H\nu(x)}>s\}}
\leq
4\pi\frac{\norm{\mu}(\bbR)}{s}.
\label{b16}
\end{equation}
Now it remains to put together \eqref{b11a}, \eqref{b12}, \eqref{b14a},
and \eqref{b16}. This yields:
\begin{multline*}
\abs{\Omega^\mu_s}
\leq
\abs{\Omega^\nu_{s/2}}
+
\abs{\Omega^f_{s/2}}
\leq
\abs{\Omega_{s/2}^\nu\cap(\cup_{\ell=1}^L 2Q_\ell)}
+
\abs{\Omega_{s/2}^\nu\cap(\cup_{\ell=1}^L 2Q_\ell)^c}
+
\abs{\Omega^f_{s/2}}
\\
\leq
\frac{(2+4\pi+2C_X)}{s/2}\norm{\mu}(\bbR),
\end{multline*}
which, after rounding up $4\pi$ to 13, gives \eqref{b8}.
\end{proof}

Next, we would like to estimate the maximal function corresponding to 
the Hilbert transform of a simple measure $\mu$ of the form \eqref{b7}.
For $x\in\bbR$ we set
\begin{align}
H_r\mu(x)
&=
\int_{\abs{y-x}\geq r}\frac{d\mu(y)}{x-y},
\quad r>0,
\label{b17}
\\
\hm\mu(x)&=\sup_{r>0}\norm{H_r\mu(x)}.
\label{b18}
\end{align}

%%%%%%%%%%%%%%%%%%%%
\begin{lemma}\label{lma.b2}
%%%%%%%%%%%%%%%%%%%%
Let $\mu$ be a simple measure of the form  \eqref{b7}.
Then $\hm\mu\in L^{1,\infty}(\bbR)$ and the estimate
\begin{equation}
\norm{\hm\mu}_{L^{1,\infty}}\leq (17592+2304C_X)\norm{\mu}(\bbR)
\label{b19}
\end{equation}
holds true.
\end{lemma}
\begin{proof}
1. 
Let us fix $r>0$ and $x\in\bbR$ and denote
$B(x,r)=(x-r,x+r)$. 
We also fix an exponent $\beta\in(0,1)$; 
we will eventually take for simplicity $\beta=1/2$, although
in principle one could optimise the estimates in $\beta$. 
Using the inequality $\abs{a+b}^\beta\leq\abs{a}^\beta+\abs{b}^\beta$,
we obtain for any $x'\in B(x,r)$:
$$
\norm{H_{3r}\mu(x)}^\beta
\leq
\norm{H\mu(x')}^\beta
+
\norm{H_{3r}\mu(x)-H\mu(x')}^\beta.
$$
Averaging over $x'\in B(x,r)$, we obtain
\begin{equation}
\norm{H_{3r}\mu(x)}^\beta
\leq
\frac1{2r}\int_{B(x,r)}\norm{H\mu(x')}^\beta dx'
+
\frac1{2r}\int_{B(x,r)}\norm{H_{3r}\mu(x)-H\mu(x')}^\beta dx'.
\label{b20}
\end{equation}
Let us estimate separately the two terms in the r.h.s. of \eqref{b20}.
For the first term, recalling the definition \eqref{b3} of $M_\beta$,
we obtain
\begin{equation}
\frac1{2r}\int_{B(x,r)}
\norm{H\mu(x')}^\beta dx'
\leq
\left[M_\beta(\norm{H\mu(\cdot)})(x)\right]^\beta.
\label{b21}
\end{equation}

2.
Let us estimate the second term in the r.h.s. of \eqref{b20}.
To this end, we set 
$$
\mu_1=\sum_{x_i\in B(x,3r)} \delta_{x_i} e_i , 
\quad
\mu_2=\sum_{x_i\notin B(x,3r)} \delta_{x_i} e_i , 
$$
and estimate the integrand in \eqref{b20} as follows:
\begin{equation}
\norm{H_{3r}\mu(x)-H\mu(x')}^\beta
\leq
\norm{H_{3r}\mu(x)-H\mu_2(x')}^\beta
+
\norm{H\mu_2(x')-H\mu(x')}^\beta.
\label{b22}
\end{equation}
Let us estimate the first term in the r.h.s. of \eqref{b22}.
Since $H_{3r}\mu(x)=H\mu_2(x)$, we have
\begin{multline}
\norm{H_{3r}\mu(x)-H\mu_2(x')}
=
\norm{H\mu_2(x)-H\mu_2(x')}
\\
=
\Norm{\int_\bbR \left(\frac1{x-y}-\frac1{x'-y}\right)d\mu_2(y)}
\leq
\int_\bbR \frac{\abs{x-x'}}{\abs{x-y}\abs{x'-y}} d\norm{\mu_2}(y).
\label{b23}
\end{multline}
Similarly to \eqref{b15a}, we have an elementary estimate 
for the integrand in \eqref{b23}:
$$
\frac{\abs{x-x'}}{\abs{x-y}\abs{x'-y}}
\leq
\frac{r}{\frac12 \abs{x-y}^2}
=
\frac{4r}{2\abs{x-y}^2}
\leq
\frac{4r}{\abs{x-y}^2+r^2}
$$
for all $x'\in B(x,r)$, $y\in (B(x,3r))^c$,
and therefore 
$$
\norm{H_{3r}\mu(x)-H\mu_2(x')}
\leq
\int_\bbR \frac{4r}{(x-y)^2+r^2}d\norm{\mu_2}(y)
\leq
\int_\bbR \frac{4r}{(x-y)^2+r^2}d\norm{\mu}(y).
$$
Finally, we use the well known estimate \cite{Garnett}
\begin{equation}
\frac1\pi\int_\bbR \frac{r}{(x-y)^2+r^2}d\norm{\mu}(y)
\leq
(M\norm{\mu})(x).
\label{b23a}
\end{equation}
This yields
$$
\norm{H_{3r}\mu(x)-H\mu_2(x')}
\leq
4\pi(M\norm{\mu})(x), 
$$
for all $x'\in B(x,r)$, and therefore
\begin{equation}
\frac{1}{2r}\int_{B(x,r)}\norm{H_{3r}\mu(x)-H\mu_2(x')}^\beta dx'
\leq
(4\pi)^\beta[M\norm{\mu}(x)]^\beta.
\label{b24}
\end{equation}

3. 
Let us estimate the average over $x'\in B(x,r)$ of the second term in 
the r.h.s. of \eqref{b22}.
Since $\mu_2-\mu=\mu_1$, we get
\begin{multline}
\frac1{2r}\int_{B(x,r)}\norm{H\mu_2(x')-H\mu(x')}^\beta dx'
=
\frac1{2r}\int_{B(x,r)}\norm{H\mu_1(x')}^\beta dx'
\\
=
\frac1{2r}\int_0^\infty 
\beta s^{\beta-1}
\abs{\{x'\in B(x,r): \norm{H\mu_1(x')}>s\}}ds.
\label{b25}
\end{multline}
Further, using Lemma~\ref{lma.b1}, 
\begin{equation}
\abs{\{x'\in B(x,r): \norm{H\mu_1(x')}>s\}}
\leq
\min\left(2r,\frac{C_1 \norm{\mu_1}(\bbR)}{s}\right),
\label{b26}
\end{equation}
where $C_1=30+4C_X$.
Next, by the definition of the maximal function $M\norm{\mu}$, we have
\begin{equation}
\norm{\mu_1}(\bbR)\leq 6r (M\norm{\mu})(x).
\label{b25a}
\end{equation}
Substituting \eqref{b26} and \eqref{b25a} into \eqref{b25}, 
after an elementary calculation we obtain
$$
\frac1{2r}
\int_0^\infty
\beta s^{\beta-1}
\min\left(2r,\frac{C_1 6r M\norm{\mu}(x)}{s}\right)
ds
\leq
(3C_1)^\beta\frac1{1-\beta}
[M\norm{\mu}(x)]^\beta.
$$
Thus, we obtain
\begin{equation}
\frac1{2r}
\int_{B(x,r)}\norm{H\mu_2(x')-H\mu(x')}^\beta dx'
\leq
(3C_1)^\beta \frac1{1-\beta}
[M\norm{\mu}(x)]^\beta.
\label{b27}
\end{equation}

4. 
Now let us combine 
\eqref{b20}, \eqref{b21}, \eqref{b24} and \eqref{b27}:
\begin{multline*}
\norm{H_{3r}\mu(x)}^\beta
\leq
\left[M_\beta(\norm{H\mu(\cdot)})(x)\right]^\beta
\\
+
(4\pi)^\beta
\left[M\norm{\mu}(x)\right]^\beta
+
(3C_1)^\beta\frac1{1-\beta}
[M\norm{\mu}(x)]^\beta.
\end{multline*}
Taking supremum over $r>0$, we obtain 
\begin{multline}
(H^\sharp \mu(x))^\beta
\leq
\left[M_\beta(\norm{H\mu(\cdot)})(x)\right]^\beta
\\
+
(4\pi)^\beta
\left[M\norm{\mu}(x)\right]^\beta
+
(3C_1)^\beta\frac1{1-\beta}
[M\norm{\mu}(x)]^\beta.
\label{b28}
\end{multline}

Denote $C_2=(4\pi)^\beta+(3C_1)^\beta\frac1{1-\beta}$, and let
\begin{align*}
\Omega(s)
&=
\{x: H^\sharp \mu(x)>s\},
\\
\Omega^{(1)}(s)
&=
\{x: M_\beta(\norm{H\mu(\cdot)})(x)>s\},
\\
\Omega^{(2)}(s)
&=
\{x: M \norm{\mu}(x)>s\}.
\end{align*}
Then \eqref{b28} implies
$$
\Omega(s)\subset\Omega^{(1)}(2^{-1/\beta}s)\cup \Omega^{(2)}(2^{-1/\beta}C_2^{-1/\beta} s).
$$
By Proposition~\ref{prp.b1} and Lemma~\ref{lma.b1}, we obtain
\begin{gather*}
\abs{\Omega^{(2)}(s)}
\leq
\frac3{s}\norm{\mu}(\bbR),
\\
\abs{\Omega^{(1)}(s)}
\leq
\frac{6^{1/\beta}}{1-\beta}\frac1s\norm{H\mu}_{L^{1,\infty}}
\leq
\frac{6^{1/\beta}}{1-\beta}\frac{C_1}s \norm{\mu}(\bbR),
\end{gather*}
and therefore
$$
\abs{\Omega(s)}
\leq
\frac{6^{1/\beta}}{1-\beta}2^{1/\beta}\frac{C_1}s\norm{\mu}(\bbR)
+
3\cdot 2^{1/\beta} C_2^{1/\beta}\frac1s\norm{\mu}(\bbR).
$$
Substituting the expressions for $C_1$ and $C_2$, taking
$\beta=1/2$, and rounding up the constants,  we obtain 
$$
\abs{\Omega(s)}
\leq
(288 C_1+12 C_2^2)\frac1s \norm{\mu}(\bbR)
\leq
(17592+2304C_X)\frac1s \norm{\mu}(\bbR),
$$
which is exactly the required relation \eqref{b19}.
\end{proof}

%%%%%%%%%%%%%%%%%%%%%%%%%%%%%%%%
\subsection{Cauchy transforms of simple measures}\label{sec.b3}
%%%%%%%%%%%%%%%%%%%%%%%%%%%%%%%%

Next, we consider the Cauchy transforms $\calC\mu$ (see \eqref{a14}) of simple measures $\mu$; 
let $T^<\mu$ be the corresponding non-tangential maximal
function defined by \eqref{a15}.

%%%%%%%%%%%%%%%%%%%%
\begin{lemma}\label{lma.b3}
%%%%%%%%%%%%%%%%%%%%
Let $\mu$ be a simple $X$-valued measure of the form \eqref{b7}. 
Then $T^<\mu$ belongs to $L^{1,\infty}(\bbR)$ and the estimate
\begin{equation}
\norm{T^<\mu}_{L^{1,\infty}}
\leq
(35274+4608 C_X)\norm{\mu}(\bbR)
\label{b29}
\end{equation}
holds true. 
\end{lemma}
\begin{proof}
1. 
Fix $\lambda\in\bbR$. Let us prove that for any $r>0$ 
\begin{equation}
\norm{\calC\mu(\lambda+x+ir)-H_{2r}\mu(\lambda)}
\leq
(2+4\pi) M\norm{\mu}(\lambda),
\label{b30}
\end{equation}
if $\abs{x}<r$. 
For simplicity of notation, let us take $\lambda=0$. Set
$z=x+ir$. 
We have:
$$
\calC\mu(z)
-
H_{2r}(0)
=
\int_\bbR \frac{d\mu(t)}{t-z}
-
\int_{\abs{t}>2r}\frac{d\mu(t)}{t}
=
\int_{\abs{t}>2r}\frac{z}{t(t-z)}d\mu(t)
+
\int_{\abs{t}\leq2r}\frac{d\mu(t)}{t-z}.
$$
Using the elementary estimate
$$
\Abs{\frac{z}{t(t-z)}}\leq 4\frac{r}{t^2+r^2}, 
\quad
\abs{x}<r, \quad \abs{t}>2r,
$$
and \eqref{b23a},
we obtain
$$
\Norm{\int_{\abs{t}>2r}\frac{z}{t(t-z)}d\mu(t)}
\leq
4\int_{\bbR}
\frac{r}{t^2+r^2}
d\norm{\mu}(t)
\leq
4\pi M\norm{\mu}(0).
$$
Finally, 
$$
\Norm{\int_{\abs{t}\leq 2r}\frac{d\mu(t)}{t-z}}
\leq
\frac1r\int_{\abs{t}\leq 2r}d\norm{\mu}(t)
\leq
2M\norm{\mu}(0),
$$
and we obtain \eqref{b30}.

2. 
By \eqref{b30}, taking supremum over $r>0$, we obtain
$$
T^<\mu(\lambda)
\leq
\hm\mu(\lambda)+(2+4\pi)M\norm{\mu}(\lambda)
$$
for all $\lambda\in\bbR$, 
and therefore 
$$
\norm{T^<\mu}_{L^{1,\infty}}
\leq
2\norm{\hm\mu}_{L^{1,\infty}}
+
2(2+4\pi)\norm{(M\norm{\mu})}_{L^{1,\infty}}.
$$
Now it remains to combine this with Proposition~\ref{prp.b1} and Lemma~\ref{lma.b2}
and work out the constants. 
\end{proof}

%%%%%%%%%%%%%%%%%%%%%%%%%%%%%%%%
\subsection{Cauchy transforms of general measures}\label{sec.b4}
%%%%%%%%%%%%%%%%%%%%%%%%%%%%%%%%

Now we are ready to handle general $X$-valued measures. 
As above, $X$ is a Banach space with a UMD property.
The lemma below is Theorem~\ref{thm.a7} with a specific value of constant:
%%%%%%%%%%%%%%%
\begin{lemma}\label{lma.b4}
%%%%%%%%%%%%%%%
Let $\mu$ be an $X$-valued measure with a bounded support on $\bbR$ 
and a bounded total variation. 
Then $T^<\mu$ belongs to $L^{1,\infty}(\bbR)$ and the estimate
\begin{equation}
\norm{T^<\mu}_{L^{1,\infty}}
\leq
(70548+9216 C_X)\norm{\mu}(\bbR)
\label{b29a}
\end{equation}
holds true. 
\end{lemma}
\begin{proof}
For $\lambda\in\bbR$ and $r>0$, denote
$$
T_r^<\mu(\lambda)
=
\sup\{\norm{\calC\mu(x+iy)}: y>r, \, \abs{x-\lambda}<y\};
$$
then 
$$
T^<\mu(\lambda)=\sup_{r>0} T_r^<\mu(\lambda). 
$$
Thus, it suffices to prove the estimate
\begin{equation}
s\cdot \abs{\{x\in\bbR: T_r^<\mu(\lambda)>s\}}
\leq 
(70548+9216 C_X)\norm{\mu}(\bbR), \quad s>0,
\label{b30a}
\end{equation}
for all $\lambda\in\bbR$ and $r>0$. 

Below we approximate the measure $\mu$ by simple measures.
For $n\in\bbN$, let $\{Q_\ell^{(n)}\}$ be the collection of all dyadic intervals of length 
$\abs{Q_\ell^{(n)}}=2^{-n}$, and let $c(Q_\ell^{(n)})$ be the center
of $Q_\ell^{(n)}$. 
Let us define the measure
\begin{equation}
\mu_n
=
\sum_\ell \delta_{c(Q_\ell^{(n)})}\mu(Q_\ell^{(n)}).
\label{b31}
\end{equation}
Since the support of $\mu$ is bounded, the sum in \eqref{b31} is finite. 
Clearly, $\mu_n$ is a simple measure of the type \eqref{b7} and 
$$
\norm{\mu_n}(\bbR)\leq \norm{\mu}(\bbR).
$$
Further, for each dyadic interval $Q$ we have
$$
\int_Qd\mu_n(x)=\int_Qd\mu(x)
$$
for all sufficiently large $n$. It follows that
$$
\Norm{
\int_\bbR \varphi(x)d\mu_n(x)
-
\int_\bbR \varphi(x)d\mu(x)}
\to 0
\quad
\text{ as $n\to\infty$}
$$
for all continuous functions $\varphi$. 
In particular, 
\begin{equation}
\Norm{\calC \mu_n(x+iy)-\calC\mu(x+iy)}
\to0
\quad
\text{ as $n\to\infty$}
\label{b32}
\end{equation}
for all $x\in\bbR$ and $y>0$. 
By analyticity, convergence in \eqref{b32} is 
uniform in $x\in\bbR$ and $y>r$ (for any fixed $r>0$).
Let us fix $r,s>0$ and choose $n$ sufficiently large so that
$$
\sup_{x\in\bbR}\sup_{y>r}
\Norm{\calC \mu_n(x+iy)-\calC\mu(x+iy)}
\leq \frac{s}2.
$$
Then 
$$
\{x\in\bbR: T_r^<\mu(x)>s\}
\subset
\{x\in\bbR: T_r^<\mu_n(x)>s/2\}.
$$
Now by Lemma~\ref{lma.b3}, 
$$
\abs{\{x\in\bbR: T_r^<\mu_n(x)>s/2\}}
\leq
(35274+4608 C_X)\frac{\norm{\mu_n}(\bbR)}{s/2}
\leq
(70548+9216 C_X)\frac{\norm{\mu}(\bbR)}{s},
$$
which proves \eqref{b30a}.
\end{proof}

%%%%%%%%%%%%%%%%%%%%%%%%%%%%%%%%%%%%%%%%%%%%%%%%%
%%%%%%%%%%%%%%%%%%%%%%%%%%%%%%%%%%%%%%%%%%%%%%%%%
\section{Boundary values of analytic Banach space valued functions}\label{sec.c}
%%%%%%%%%%%%%%%%%%%%%%%%%%%%%%%%%%%%%%%%%%%%%%%%%
%%%%%%%%%%%%%%%%%%%%%%%%%%%%%%%%%%%%%%%%%%%%%%%%%

%%%%%%%%%%%%%%%%%%%%%%%%%%%%%%%%%%%%%%%
\subsection{Boundary values of Cauchy transforms}\label{sec.c1}
%%%%%%%%%%%%%%%%%%%%%%%%%%%%%%%%%%%%%%%

First we need a general statement about existence of boundary values
of bounded Banach space analytic functions. 
This statement is known (in much greater generality, see e.g. \cite{Bu})
but for completeness below we give a simple proof. 

%%%%%%%%%%%%%%%%%%%%%%%
\begin{proposition}\cite{Bu}\label{prp.c1}
%%%%%%%%%%%%%%%%%%%%%%%
Let $X$ be a reflexive Banach space with a separable 
dual $X^*$, and let $F:\bbD\to X$ be a bounded analytic function. 
Then the non-tangential limit
\begin{equation}
F(e^{i\theta})
=
\lim_{y\to+0} F(e^{i\theta}(1-ix-y)), \quad  \abs{x}<y,
\label{c0}
\end{equation}
exists for a.e. $\theta\in[0,2\pi)$ in the norm of $X$. 
\end{proposition}

\begin{proof}
1. 
Since $F$ is bounded, we may assume without loss of generality that $\norm{F(z)}\leq 1$ for all $\abs{z}<1$.
Let $\ell\in X^*$; then $\ell(F(z))$ is a scalar bounded analytic function, and therefore the non-tangential
boundary values
\begin{equation}
F_\ell(e^{i\theta})
=
\lim_{y\to+0} \ell(F(e^{i\theta}(1-ix-y))), \quad  \abs{x}<y,
\label{c1}
\end{equation}
exist for a.e. $\theta\in[0,2\pi)$ and $\ell(F(z))$ can be represented
as a Poisson integral of these boundary values: 
\begin{equation}
\ell(F(z))
=
\int_0^{2\pi}
P(z,e^{i\theta})F_\ell(e^{i\theta})d\theta, 
\quad
\abs{z}<1.
\label{c2}
\end{equation}
Let $X'\subset X^*$ be a countable dense set; then there exists
a set $Z\subset[0,2\pi)$ of full Lebesgue measure such that 
the limit \eqref{c1} exists for all $\ell\in X'$ and all $\theta\in Z$. 
Further, for all $\theta\in Z$, we have
$$
\abs{F_\ell(e^{i\theta})}\leq \norm{\ell}_{X^*}, 
\quad
\forall \ell\in X',
\quad
$$
and therefore we can extend the linear map $X'\ni\ell\mapsto F_\ell(e^{i\theta})$ 
to a bounded linear functional on $X^*$. By reflexivity, it follows that for every 
$\theta\in Z$ there exists an element $\wt F(e^{i\theta})\in X$ such that 
$\norm{\wt F(e^{i\theta})}\leq 1$ and $F_\ell(e^{i\theta})=\ell(\wt F(e^{i\theta}))$.

2. 
Let us check that $\wt F$ is in $L^1(\bbT, X)$. 
Since $\wt F$ is bounded, we only need to check that 
the map $\theta\mapsto \wt F(e^{i\theta})\in X$ is (Borel) measurable.
For any $\psi\in X$ and $a>0$, we have 
$$
\{\theta\in[0,2\pi): \norm{\wt F(e^{i\theta})-\psi}\leq a\}
=
\bigcap_{\ell\in X'}\{\theta\in[0,2\pi): \abs{F_\ell(e^{i\theta})-\ell(\psi)}\leq a\norm{\ell}_{X^*}\},
$$
and the r.h.s. is a measurable set by the measurability of $F_\ell$ for all $\ell$. 
Thus, the pre-image by $\wt F$ of any closed ball is measurable; from here it is easy to 
derive the measurability of $\wt F$.

3. By the previous step, we can integrate $\wt F(e^{i\theta})$, and using \eqref{c2}, we obtain
\begin{equation}
F(z)
=
\int_0^{2\pi}
P(z,e^{i\theta})\wt F(e^{i\theta})d\theta, 
\quad
\abs{z}<1.
\label{c3}
\end{equation}
Now, following the standard argument and using the density 
of continuous functions in $L^1(\bbT,X)$, one checks that 
a.e. $\theta\in[0,2\pi)$ is a Lebesgue point of $\wt F(e^{i\theta})$. 
From here and the representation \eqref{c3}
in a standard way one proves that the non-tangential limit
\eqref{c0} exists and equals $\wt F(e^{i\theta})$ at every 
Lebesgue point. 
\end{proof}

%%%%%%%%%%%%%%%%%
\begin{lemma}\label{lma.c2}
%%%%%%%%%%%%%%%%%
Let $X$ be a reflexive Banach space with a separable 
dual $X^*$, and suppose that $X$ possesses the UMD property.
Let $\mu$ be an $X$-valued measure on $\bbR$ with a finite total variation, 
and let $C\mu$ be the Cauchy transform of $\mu$, see \eqref{a14}.
Then the non-tangential limit 
\begin{equation}
\calC \mu(\lambda)
=
\lim_{y\to+0}\calC\mu(\lambda+x+iy), 
\quad \abs{x}<y,
\label{c4}
\end{equation}
exists for a.e. $\lambda\in\bbR$ in the norm of $X$.
\end{lemma}
\begin{proof}
First note that the existence of the limit \eqref{c4} is a local property of $\mu$, 
and therefore without loss of generality we may assume that $\supp \mu\subset [-1,1]$. 
Next, let $T^<\mu$ be the non-tangential maximal function \eqref{a15}, 
and for $s>0$ let 
$$
\Omega_s=\{x \in\bbR: (T^<\mu)(x)>s\}.
$$
By Theorem~\ref{thm.a7}, we have $\abs{\Omega_s}\leq \frac{C}{s}$
for all $s>0$. 
Let $\widehat \Omega_s$ be an open set such that $\Omega_s\subset \widehat \Omega_s$ 
and $\abs{\widehat \Omega_s}\leq \frac{2C}{s}$. 
It suffices to prove that for all $s>0$ the limit \eqref{c4} exists for a.e. $\lambda\in(\widehat \Omega_s)^c$. 

Below we essentially repeat the construction of the classical Privalov uniqueness theorem 
(see e.g. \cite[Section III D]{Koosis}). 
Fix $s>0$; the set $\widehat \Omega_s$ can be written as a union of open intervals 
$\cup I_n$. Let $c(I_n)$ be the center of $I_n$ and let $\Delta_n$ be the closed isosceles 
triangle with base $\overline{I_n}$ and $\pi/2$ angle opposite to $\overline{I_n}$: 
$$
\Delta_n=
\{x+iy: 0\leq y\leq \tfrac12\abs{I_n}, \ \abs{x-c(I_n)}\leq \abs{y-\tfrac12\abs{I_n}}\}.
$$
Consider the domain 
$$
D=\{z: \Im z>0, \abs{z}<2, z\notin \Delta_n\  \forall n\}.
$$
By construction, $D$ is open, simply connected, 
the boundary of $D$ is a rectifiable Jordan curve, 
and $\norm{\calC\mu(z)}\leq s$ for all $z\in D$. 
Let $\varphi$ be a conformal map of the unit disk $\bbD$ 
onto $D$ and put $F(z)=\calC\mu(\varphi(z))$ for $\abs{z}<1$. 
Then by Proposition~\ref{prp.c1}, the non-tangential boundary
values of $F$ exist for almost all points on the unit circle. 
Since the boundary of $D$ is a rectifiable Jordan curve,
it follows (see e.g. \cite[Section~II C]{Koosis}) that the 
non-tangential limits of $\calC\mu(z)$ exist for almost every
point on the boundary of $D$. In particular, these limits 
exist almost everywhere on $(\widehat\Omega_s)^c$, 
as required. 
\end{proof}

%%%%%%%%%%%%%%%%%%%%%%%%%%%%%%%%%%%%%%%
\subsection{Boundary values of $B_0$ and $B_1$ under global assumptions}\label{sec.c2}
%%%%%%%%%%%%%%%%%%%%%%%%%%%%%%%%%%%%%%%
In this subsection and in the next one, we prove the existence of 
the boundary values of the sandwiched resolvents  $B_0(z)$, $B_1(z)$ (see \eqref{a11})
in an appropriate sense. 
In order to make our presentation more readable, 
in this subsection we consider the simplest case of 
Theorems~\ref{thm.a3} and \ref{thm.a4} under the additional 
assumption that the measures $\nu_0$ and $\nu_1$ in \eqref{a4}, \eqref{a5} 
are finite rather than $\sigma$-finite. 
In the next subsection, by using  standard localisation 
arguments, we extend these results to the case of the local 
assumptions of Theorems~\ref{thm.a5} and \ref{thm.a6}.

%%%%%%%%%%%%%%%%%
\begin{theorem}\label{thm.c3}
%%%%%%%%%%%%%%%%%
Assume the hypothesis of Theorem~\ref{thm.a3} 
and suppose in addition that the measure $\nu_0$ in \eqref{a4} is finite. 
Then the limits $B_0(\lambda+i0)$, $B_1(\lambda+i0)$
exist for a.e. $\lambda\in\bbR$ in the  norm of $\Sch_p$. 
\end{theorem}
\begin{proof}
By the spectral theorem, the function $B_0(z)$ is the 
Cauchy transform of the $\Sch_p$-valued measure 
$\mu_0$ given by \eqref{a12}. 
Assumption \eqref{a4} ensures that $\mu_0$ has a finite total variation. 
Thus, the existence of the non-tangential limits $B_0(\lambda+i0)$
in the norm of $\Sch_p$ 
follows from Lemma~\ref{lma.c2} with $X=\Sch_p$. 

Let us consider the function $B_1(z)$. 
First let us check the identity \eqref{a17}.
Using the standard resolvent identity 
\begin{equation}
R_1(z)=R_0(z)-R_1(z)G^*JGR_0(z)
=R_0(z)-R_0(z)G^*JGR_1(z),
\label{c5}
\end{equation}
we obtain
\begin{equation}
(I-B_1(z)J)(I+B_0(z)J)
=
(I+B_0(z)J)(I-B_1(z)J)
=I.
\label{c7}
\end{equation}
Thus, $I+B_0(z)J$ has a bounded inverse for all $\Im z\not=0$. 
Next, similarly to \eqref{c5}, we get
$$
(I+B_0(z)J)B_1(z)=B_0(z),
$$
and thus we obtain \eqref{a17}. 
Since the limits $B_0(\lambda+i0)$ exist in $\Sch_p$ for a.e. $\lambda\in\bbR$,
it suffices to check that the operator $I+B_0(\lambda+i0)J$
is invertible for a.e. $\lambda\in\bbR$. 
In order to do this, we employ the argument of \cite[Theorem~1.8.5]{Yafaev}. 
Let $q\geq p$ be an integer;
we make use of the regularised determinant $\Det_q$, see e.g. \cite[Section~1.7]{Yafaev}.
We only need two properties of $\Det_q$:
\begin{enumerate}[(a)]
\item
$\Det_q(I+A)$ is an analytic function of $A\in\Sch_q$; 
\item
if $A\in\Sch_q$, then $I+A$ is invertible if and only if $\Det_q(I+A)\not=0$. 
\end{enumerate}
As $\Sch_p\subset \Sch_q$, the above properties of course also apply to $A\in\Sch_p$. 
Consider the function 
$$
d(z)=\Det_q(I+B_0(z)J), \quad \Im z>0.
$$
The function $d(z)$ is analytic and non-vanishing in $\bbC_+$. 
By the previous step of the proof, the non-tangential limits $d(\lambda+i0)$ 
exist for a.e. $\lambda\in\bbR$. By Privalov's uniqueness theorem 
(see e.g. \cite[Section III D]{Koosis}), it follows that $d(\lambda+i0)\not=0$ 
for a.e. $\lambda\in\bbR$. Thus, the operator $I+B_0(\lambda+i0)J$
is invertible for a.e. $\lambda\in\bbR$, as required. 
\end{proof}

%%%%%%%%%%%%%%%%%
\begin{theorem}\label{thm.c4}
%%%%%%%%%%%%%%%%%
Assume the hypothesis of Theorem~\ref{thm.a4}
and suppose in addition that the measures $\nu_0$ and $\nu_1$ in \eqref{a5} are finite. 
Then for any $\psi\in\calH$, there exists a set $Z_\psi\subset \bbR$ 
of full Lebesgue measure such that for all $\lambda\in Z_\psi$
the limits 
\begin{equation}
\lim_{\eps\to+0}
B_j(\lambda+i\eps)\psi,
\quad
j=0,1,
\label{c6}
\end{equation}
exist in the norm of $\calH$. 
\end{theorem}
\begin{proof}
Follows by applying Lemma~\ref{lma.c2} to the $\calH$-valued measures 
\eqref{a18}. 
\end{proof}

%%%%%%%%%%%%%%%%%%%%%%%%%%%%%%%%%%%%%%%
\subsection{Boundary values of $B_0$ and $B_1$ under local assumptions}\label{sec.c3}
%%%%%%%%%%%%%%%%%%%%%%%%%%%%%%%%%%%%%%%

%%%%%%%%%%%%%%%%%
\begin{theorem}\label{thm.c6}
%%%%%%%%%%%%%%%%%
Assume the hypothesis of Theorem~\ref{thm.a5}.
Then the limits $B_0(\lambda+i0)$, $B_1(\lambda+i0)$
exist for a.e. $\lambda\in\Delta$ in the operator norm. 
\end{theorem}
\begin{proof}
1.
For $j=0,1$, let us write
$$
R_j(z)=R_j(z)E_{H_0}(\Delta)+R_j(z)E_{H_0}(\Delta^c);
$$
this gives rise to the decomposition 
\begin{equation}
B_j(z)=B_j^{(1)}(z)+B_j^{(2)}(z)
\label{c9}
\end{equation}
with 
$$
B_j^{(1)}(z)
=
GR_j(z)E_{H_j}(\Delta)G^*,
\quad
B_j^{(2)}(z)
=
GR_j(z)E_{H_j}(\Delta^c)G^*.
$$
It is clear that the operators $B_j^{(2)}(z)$ 
are analytic in $z\in\bbC\setminus(\bbR\setminus\Delta)$. 

Consider the operator $B_0(z)$.
Representation \eqref{c9} reduces the problem to the existence
of limits $B_0^{(1)}(\lambda+i0)$. 
Similarly to the proof of Theorem~\ref{thm.c3}, we define the measure 
$\mu_0$ on $\Delta$ by
$$
\mu_0(\delta)
=
(GE_{H_0}(\delta))(GE_{H_0}(\delta))^*, 
\quad
\delta\subset\Delta,
$$
and find that $B_0^{(1)}(z)=\calC\mu_0(z)$, $\Im z>0$. 
By hypothesis, the $\Sch_p$-valued measure $\mu_0$ 
has a finite total variation. 
Thus, by Lemma~\ref{lma.c2} with $X=\Sch_p$, 
the non-tangential limits $B_0^{(1)}(\lambda+i0)$ 
exist in the norm of $\Sch_p$ for a.e. $\lambda\in\Delta$.

2. 
Before considering the boundary values of the operator $B_1(z)$, 
we need to discuss the properties of the operator $B_0^{(2)}(z)$.
First note that, by assumption, $GR_0(z)$ is compact and therefore
$B_0^{(2)}(z)$ is compact. Next, let us check that the operator
$I+B_0^{(2)}(z)J$ is invertible for all $\Im z\not=0$. 
Assume that 
\begin{equation}
\psi+B_0^{(2)}(z)J\psi=0
\label{c14}
\end{equation}
for some element $\psi$.
Taking the inner product with $J\psi$ and evaluating the imaginary part,
we obtain
$$
\Im (B_0^{(2)}(z)J\psi,J\psi)=0,
$$
which can be written as
$$
\Im (R_0(z)\varphi,\varphi)=0,
\quad
\varphi=E_{H_0}(\Delta^c)G^*J\psi.
$$
Since the kernel of $\Im R_0(z)$ is trivial, we
get $\varphi=0$. It follows that $B_0^{(2)}(z)J\psi=0$ and so by \eqref{c14} we get $\psi=0$,
as required.

3. 
Consider the operator $B_1(z)$.
As in the proof of Theorem~\ref{thm.c3}, formula \eqref{a17}
reduces the proof to checking that
the operator $I+B_0(\lambda+i0)J$ is invertible for a.e. $\lambda\in\Delta$. 
We will use the decomposition \eqref{c9}.

As discussed above, $B_0^{(2)}(z)$ is compact and analytic in 
$z\in\bbC\setminus(\bbR\setminus\Delta)$. 
Let $\calQ$ be the set of $\lambda\in\Delta$ such that the equation
$$
\psi+B_0^{(2)}(\lambda)J\psi=0
$$
has a non-trivial solution $\psi\not=0$. By the analytic Fredholm 
alternative, the set $\calQ$ is countable and the elements of $\calQ$ 
may accumulate only to the endpoints of $\Delta$. 
We obtain that $(I+B_0^{(2)}(z)J)^{-1}$ is analytic in $z\in\bbC_+$ and has 
non-tangential limit values as $z\to\lambda$ for all $\lambda\in\Delta\setminus\calQ$. 

Denote 
$$
\wt B_0(z)=(I+B_0^{(2)}(z)J)^{-1}B_0^{(1)}(z)J, 
\quad \Im z>0.
$$
By the above analysis and by step 1 of the proof, the operator $\wt B_0(z)$ 
belongs to $\Sch_p$ and the non-tangential limits $\wt B_0(\lambda+i0)$ 
exist for a.e. $\lambda\in\Delta$ in the norm of $\Sch_p$. 
Using the analytic function $\Det_q(I+\wt B_0(z))$ and Privalov's uniqueness
theorem just as in the proof of Theorem~\ref{thm.c3}, we conclude 
that the operator $I+\wt B_0(\lambda+i0)$ is invertible for a.e. $\lambda\in\Delta$. 
Finally, we have 
$$
I+B_0(z)J=(I+B_0^{(2)}(z)J)(I+\wt B_0(z));
$$
since both $I+B_0^{(2)}(\lambda+i0)J$ and $I+\wt B_0(\lambda+i0)$ exist and are invertible
for a.e. $\lambda\in\Delta$, we obtain the required statement. 
\end{proof}

%%%%%%%%%%%%%%%%%
\begin{theorem}\label{thm.c7}
%%%%%%%%%%%%%%%%%
Assume the hypothesis of Theorem~\ref{thm.a6}. 
Then for any $\psi\in\calH$, there exists a set $Z_\psi\subset \Delta$ 
of full Lebesgue measure such that for all $\lambda\in Z_\psi$
the limits 
\begin{equation}
\lim_{\eps\to+0}
B_j(\lambda+i\eps)\psi,
\quad
j=0,1,
\label{c15}
\end{equation}
exist in the norm of $\calH$. 
\end{theorem}
\begin{proof}
Representation \eqref{c9} reduces the question to the proof 
of existence of the limits $B_j^{(1)}(\lambda+i0)\psi$. 
The existence of these limits 
follows by applying Lemma~\ref{lma.c2} to the $\calH$-valued measures 
\eqref{a18} on $\Delta$. 
\end{proof}

%%%%%%%%%%%%%%%%%%%%%%%%%%%%%%%%%%%%%%%%%%%%%%%%%
%%%%%%%%%%%%%%%%%%%%%%%%%%%%%%%%%%%%%%%%%%%%%%%%%
\section{Proof of the main results}\label{sec.d}
%%%%%%%%%%%%%%%%%%%%%%%%%%%%%%%%%%%%%%%%%%%%%%%%%
%%%%%%%%%%%%%%%%%%%%%%%%%%%%%%%%%%%%%%%%%%%%%%%%%

It is clear that Theorem~\ref{thm.a3} follows from Theorem~\ref{thm.a5}, 
and Theorem~\ref{thm.a4} follows from Theorem~\ref{thm.a6}. 
Thus, it remains to prove  Theorems~\ref{thm.a5} and \ref{thm.a6}.
Let $H_1=H_0+G^*JG$, where 
the operators $G$ and $J=J^*$ are bounded and $GR_0(z)$
is compact. 
We use the following statement from abstract scattering theory:
%%%%%%%%%%%%%%%%%%%%
\begin{proposition}\cite[Section~5]{Yafaev}\label{prp.d1}
%%%%%%%%%%%%%%%%%%%%
Let $H_j$, $B_j(z)$, $j=0,1$ be as defined above. 
Let $\Delta\subset \bbR$ be an open interval. Assume that for 
a.e. $\lambda\in \Delta$, the \emph{weak} limits
\begin{equation}
\wlim_{\eps\to+0} \Im B_j(\lambda+i\eps), 
\quad
j=0,1,
\label{d1}
\end{equation}
exist; and assume that for any $\psi\in\calH$, there exists
a set $Z_\psi\subset\Delta$ of full Lebesgue measure such that
for all $\lambda\in Z_\psi$, the limits
\begin{equation}
\lim_{\eps\to+0}B_j(\lambda+i\eps)\psi, 
\quad
j=0,1,
\label{d2}
\end{equation}
exist in the norm of $\calH$. 
Then the local wave operators $W_\pm(H_1,H_0;\Delta)$ 
exist and are complete. 
\end{proposition}

Theorem~\ref{thm.a5} follows immediately from Theorem~\ref{thm.c6} and
Proposition~\ref{prp.d1}.

\begin{proof}[Proof of Theorem~\ref{thm.a6}]
Existence of limits \eqref{d2} is given by Theorem~\ref{thm.c7}. 
Thus, it only remains to check the existence of the weak limits \eqref{d1}. 
It suffices to consider the case $j=0$; we use the notation 
$\mu_0(\delta)=GE_{H_0}(\delta)G^*$ for $\delta\subset \Delta$. 
Let us prove that for a.e. $\lambda\in\Delta$, the weak limit
$$
\wlim_{\eps\to+0}\tfrac1{2\eps} \mu_0(\lambda-\eps,\lambda+\eps)
=:
\mu_0'(\lambda)
$$
exists. Then, by a standard argument, the weak limits \eqref{d1} exist
on the same set of $\lambda\in\Delta$ and are equal to $\pi\mu_0'(\lambda)$. 

For any $\psi,\varphi\in\calH$, there exists a set $Z_{\psi,\varphi}\subset \Delta$
of full Lebesgue measure such that for all $\lambda\in Z_{\psi,\varphi}$ the limit
\begin{equation}
B_\lambda (\psi,\varphi):=\lim_{\eps\to+0}\tfrac1{2\eps}(\mu_0(\lambda-\eps,\lambda+\eps)\psi,\varphi)
\label{d3}
\end{equation}
exists. 
Let $\calH_0$ be a dense countable subset of $\calH$. 
Then by the above, we can choose a subset $Z\subset \Delta$ of full Lebesgue 
measure such that for all $\lambda\in Z$ and all $\psi,\varphi\in \calH_0$ the limit
\eqref{d3} exists. 
Next, there exists a set $Z_0\subset\Delta$ of full Lebesgue measure such that
for all $\lambda\in Z_0$  the limit
$\lim_{\eps\to+0}\frac1{2\eps}\nu_0(\lambda-\eps,\lambda+\eps)$
exists and therefore the Hardy-Littlewood maximal function 
$M\nu_0(\lambda)$ (see \eqref{b2}) is finite. 
Let $\lambda\in Z\cap Z_0$; by the hypothesis 
\eqref{a10}, we get
\begin{multline*}
\abs{B_\lambda(\psi,\varphi)}
\leq 
\norm{\psi}\norm{\varphi}
\limsup_{\eps\to+0}\tfrac1{2\eps}\norm{\mu_0(\lambda-\eps,\lambda+\eps)}
\\
\leq
\norm{\psi}\norm{\varphi}
\limsup_{\eps\to+0}\tfrac1{2\eps}\nu_0(\lambda-\eps,\lambda+\eps)
\leq
\norm{\psi}\norm{\varphi} M\nu_0(\lambda), 
\end{multline*}
and so $B_\lambda(\psi,\varphi)=(\mu_0'(\lambda)\psi,\varphi)$ 
for some bounded operator $\mu_0'(\lambda)$ in $\calH$. 
Further, 
$$
\Abs{
\tfrac1{2\eps}(\mu(\lambda-\eps,\lambda+\eps)\psi,\varphi)
-
(\mu_0'(\lambda)\psi,\varphi)}
\leq
2M\nu_0(\lambda)\norm{\psi}\norm{\varphi}
$$
and therefore one can extend the limiting relation
\eqref{d3} from $\psi,\varphi\in \calH_0$ to 
$\psi,\varphi\in \calH$. 
This completes the proof.
\end{proof}

\section*{Acknowledgements}
The authors are grateful to S.~V.~Kislyakov for 
expert advice on Banach space valued analytic functions and to D.~R.~Yafaev for useful discussions. 
This work was completed during the authors' stay at MSRI Berkeley; 
the authors are grateful to MSRI for hospitality. 
A.P.'s visit to MSRI was supported by the Royal Society International Exchanges
scheme.

\end{document}